\documentclass[11pt,a4paper]{amsart}
\usepackage{mathrsfs}
\usepackage{syntonly}
\usepackage{amsmath}
\usepackage{amsthm}
\usepackage{amsfonts}
\usepackage{amssymb}
\usepackage{latexsym}
\usepackage{amscd,amssymb,amsopn,amsmath,amsthm,graphics,amsfonts,mathrsfs,accents,enumerate,verbatim,calc}
\usepackage[dvips]{graphicx}
\usepackage[colorlinks=true,linkcolor=red,citecolor=blue]{hyperref}
\usepackage[all]{xy}
\usepackage{enumitem}

\date{}
\hypersetup
{colorlinks=true,linkcolor=black}
\allowdisplaybreaks[4] \footskip=15pt
\renewcommand{\uppercasenonmath}[1]{}

\numberwithin{equation}{section} \theoremstyle{plain}
\newtheorem{lem}{Lemma}[section]
\newtheorem{cor}[lem]{Corollary}
\newtheorem{prop}[lem]{Proposition}
\newtheorem{thm}[lem]{Theorem}

\newtheorem{definition}[lem]{Definition}
\newtheorem{Ex}[lem]{Example}
\newtheorem{Quest}[lem]{Question}
\newtheorem{Property}[lem]{Property}
\newtheorem{Properties}[lem]{Properties}
\newtheorem{Subprops}{}[lem]
\newtheorem{Para}[lem]{}

\newtheorem{fact}[lem]{Fact}

\newtheorem{remark}[lem]{Remark}
\newtheorem{rem}[lem]{Remark}

\newenvironment{ex}{\begin{Ex}\rm}{\end{Ex}}

\newtheorem*{ack*}{ACKNOWLEDGEMENTS}

\newcommand{\pf}{\noindent\begin {proof}}
\newcommand{\epf}{\end{proof}}

\newcommand{\mm}{\mathfrak{m}}

\newcommand{\s}{\stackrel}

\newcommand{\Hom}{{\rm Hom}}

\newcommand{\amp}{{\rm amp}}

\begin{document}
\begin{center}
{\Large  \bf G-dimensions for DG-modules over commutative DG-rings}

\vspace{0.5cm} Jiangsheng Hu, Xiaoyan Yang\footnote{Corresponding author}, Rongmin Zhu\\
\medskip
\end{center}
\bigskip
\centerline { \bf  Abstract}
\leftskip10truemm \rightskip10truemm \noindent
We define and study a notion of G-dimension
for DG-modules over a non-positively graded commutative noetherian DG-ring $A$. Some criteria for the finiteness of the G-dimension of a DG-module are given by applying a DG-version of projective resolution introduced by Minamoto [Israel J. Math. 245 (2021) 409-454]. Moreover, it is proved that the finiteness of G-dimension characterizes
the local Gorenstein property of $A$. Applications go in three directions. The first is to establish the connection between G-dimensions and the little finitistic dimensions of $A$. The second is to characterize Cohen-Macaulay and Gorenstein DG-rings by
 the relations between the class of maximal local-Cohen-Macaulay DG-modules and a special G-class of DG-modules. The third is to extend the classical Buchweitz-Happel Theorem and its inverse from commutative noetherian local rings to the setting of commutative noetherian local DG-rings. Our method is somewhat different from classical commutative ring.
\leftskip10truemm \rightskip10truemm \noindent
\\[2mm]

\leftskip0truemm \rightskip0truemm
{\footnotesize\tableofcontents\label{contents}}
\section{Introduction }\label{pre}

The Gorenstein dimension, or G-dimension, for finitely generated modules over a commutative noetherian ring was introduced by Auslander \cite{Auslander1} and was developed deeply by Auslander and Bridger \cite{AB1969}. The reason for the name is that G-dimension characterizes Gorenstein local rings exactly as projective dimension characterizes regular local rings. With that as a start, G-dimension has played an important role in singularity theory \cite{Buchweitz,Yoshino}, cohomology theory of commutative rings \cite{AM2002,chr} and representation theory of Artin algebras \cite{AR1991,RZ}. Over a general ring, Enochs and Jenda \cite{EJ1995} defined Gorenstein projective dimension for arbitrary modules.  For finitely generated modules over commutative noetherian rings it coincides with the Auslander and Bridger's G-dimension. For left modules over arbitrary associative rings, Holm \cite{Holm} proved that the new concept has the desired properties.

In a different direction, Yassemi \cite{Yassemi} studied G-dimension for complexes over a commutative noetherian local ring through a consistent use of the RHom-functor of complexes and the related category of reflexive complexes, while Christensen \cite{chr} went straight for the throat and gave the definition in terms of resolutions. Thus, the two definitions are equivalent, and they are both rooted in a result~---~due to Foxby --- saying that a finitely generated module has finite G-dimension if and only if it
is reflexive as a complex in the sense defined in \cite[Definition 2.4]{Yassemi} (see also \cite[Definition 2.1.6]{chr}).

Despite the great success of the G-dimension in commutative noetherian rings, until now it was
completely missing from higher algebra.
The aim of this paper is to introduce and study G-dimensions in
derived commutative algebra.  More specifically, we work with non-positively graded commutative DG-rings $A=\bigoplus_{i=-\infty}^{0}A^{i}$ with a differential of degree $+1$. These include the normalizations of simplicial commutative rings.

Given a commutative DG-ring $A$, the derived category of DG-modules over $A$ will be denoted by $\mathrm{D}(A)$. Its full subcategory consisting of DG-modules with bounded cohomology will be denoted by $\mathrm{D}^{\mathrm{b}}(A)$. We denote by $\mathrm{D}^{\mathrm{b}}_{\mathrm{f}}(A)$ the full triangulated subcategory of $\mathrm{D}^{\mathrm{b}}(A)$ consisting of DG-modules with finitely generated cohomology. For
a DG-module $M$, we set $\mathrm{inf}{M} = \inf\{n\hspace{0.03cm}|\hspace{0.03cm}\mathrm{H}^{n}(M) \neq 0\},\ \mathrm{sup}{M} = \sup\{n\hspace{0.03cm}|\hspace{0.03cm} \mathrm{H}^{n}(M) \neq 0\}$
and $\amp{M} = \mathrm{sup}{M} -\mathrm{inf}{M}$.

Motivated by the G-dimension for complexes over a commutative noetherian ring defined in \cite{chr,Yassemi}, we introduce the definition of G-dimension for DG-modules.

\begin{definition}\label{def:1.2} Let $A$ be a noetherian DG-ring with $\mathrm{amp}A<\infty$.
\begin{enumerate}
\item A DG-module $X\in \mathrm{D}^{\mathrm{b}}_{\mathrm{f}}(A)$ is said to be \emph{reflexive} if $\mathrm{RHom}_{A}(X,A)\in{\mathrm{D}^{\mathrm{b}}_{\mathrm{f}}(A)}$ and the morphism $X\rightarrow\mathrm{RHom}_{A}(\mathrm{RHom}_{A}(X,A),A)$ is {an isomorphism} in $\mathrm{D}^{\mathrm{b}}_{\mathrm{f}}(A)$.
\item For a reflexive DG-module $X$, we define the \emph{G-dimension} of $X$, denoted by $\mathrm{G}\textrm{-}\mathrm{dim}_AX$, by the formula
$$\mathrm{G}\textrm{-}\mathrm{dim}_AX=\mathrm{sup}\mathrm{RHom}_{A}(X,A).$$
If $X$ is not reflexive, we say that it has infinite G-dimension and write $\mathrm{G}\textrm{-}\mathrm{dim}_AX=\infty$.
\end{enumerate}
\end{definition}

 Let $A$ be an ordinary ring. If $X$ is an $A$-module, then one can show that this definition coincides with the usual definition of the G-dimension of $X$ defined by Auslander in \cite{Auslander1}; moreover, if $X$ is an $A$-complex, then the G-dimension of $X$ defined here is just the definition defined by Yassemi in \cite[Definition 2.8]{Yassemi} or by Christensen in \cite[Definition 2.3.3 and Theorem 2.3.7]{chr}.

 We denote by $\mathcal{R}(A)$ for the full subcategory of $\mathrm{D}(A)$ consisting of reflexive DG-modules. A DG-module $X\in\mathcal{R}(A)$ is said to be in the \emph{G-class} $\mathcal{G}$ if either $\mathrm{G}\textrm{-}\mathrm{dim}_AX=-\mathrm{sup}X$, or $X=0$, and denote
 by $\mathcal{G}_0$ the full subcategory of $\mathcal{G}$ consisting of objects $G$ such that either $\mathrm{amp}G\geq\mathrm{amp}A$ and $\mathrm{G}\textrm{-}\mathrm{dim}_AG=-\mathrm{sup}G=0$, or $G=0$ (see Definition \ref{def:1.2'}). If we write $\mathcal{P}\subseteq \mathrm{D}(A)$ for the full subcategory of direct summands of a finite direct sums of $A$, i.e. $\mathcal{P}=\mathrm{add}A$, then it is easy to check  $\mathcal{P}\subseteq\mathcal{G}_0$. If $A$ is an ordinary ring and $X$ an $A$-module, then $\mathcal{G}=\mathcal{G}_0$ is exactly the class of modules of G-dimension zero.

Following \cite{Mi18}, for any $0\not\simeq X\in \mathrm{D}^+(A)$,
a \emph{sppj morphism} $f:P\rightarrow X$ is a morphism in $\mathrm{D}(A)$ such that $P\in \mathrm{Add}A[-\mathrm{sup}X]$ and
the morphism $\mathrm{H}^{\mathrm{sup}X}(f)$ is surjective, where $\mathrm{Add}A$ is the full subcategory of direct summands of a direct sum of $A$.
 Moreover, a \emph{sppj resolution} $P_{\bullet}$ of $X$ is a sequence of exact triangles $X_{i+1}\stackrel{g_{i+1}}\rightarrow P_i\stackrel{f_i}\rightarrow X_i\rightsquigarrow$
such that $f_i$ is a sppj morphism for $i\geq 0$ with $X_0:=X$.
If $A$ is noetherian and $X\in\mathrm{D}^{\mathrm{b}}_{\mathrm{f}}(A)$ then we can choose $P_i$ in $\mathcal{P}$ by \cite[Proposition 2.26]{Mi18}.

We set $\mathrm{injdim}_{A}A:=\mathrm{inf}\{n\in\mathbb{Z}\mid\mathrm{Ext}^i_A(N,A)=0\ \textrm{for\ any}\ N\in\mathrm{D}^\mathrm{b}(A)\ \textrm{and}\ i>n-\mathrm{inf}N\}$. Recall from \cite{FIJ2003,FJ2003} that a commutative noetherian local DG-ring $(A,\bar{\mathfrak{m}},\bar{k})$ is called \emph{Gorenstein} if $\mathrm{amp}A<\infty$ and $\mathrm{injdim}_{A}A<\infty$.

Recall that for full subcategories $\mathcal{X},\mathcal{Y}\subseteq \mathrm{D}(A)$, we define $\mathcal{X}\ast\mathcal{Y}$
to be the full subcategory consisting $Z\in \mathrm{D}(A)$ which fits into an exact triangle $X\rightarrow Z\rightarrow Y\rightsquigarrow $ with $X\in \mathcal{X}$ and $Y\in \mathcal{Y}$.

Now, our main result can be stated as follows, which coveys that the standard techniques for classical commutative algebra can be generalized to commutative DG-algebras with appropriate modifications.

\begin{thm}\label{thm:main} Let $(A,\bar{\mathfrak{m}},\bar{k})$ be a commutative noetherian local DG-ring with $\mathrm{amp}A<\infty$.
\begin{enumerate}
\item[(1)] If $0\not\simeq X\in{\mathrm{D}^{\mathrm{b}}_{\mathrm{f}}(A)}$ with $\mathrm{amp}X\geq\mathrm{amp}A$ and $n$ is a natural number, then the following are equivalent:
\begin{enumerate}
\item[(i)] $\mathrm{G}\textrm{-}\mathrm{dim}_AX\leq n-\mathrm{sup}X$;

\item[(ii)] $X$ has a sppj resolution $P_\bullet$ such that $X_e\in\mathcal{G}_0[-\mathrm{sup}X_e]$;

\item[(iii)]  $X$ belongs to $\mathcal{P}[-\mathrm{sup}X]\ast\cdots\ast\mathcal{P}[-\mathrm{sup}X+n-1]\ast\mathcal{G}_0[-\mathrm{sup}X+n]$;

\item[(iv)]  $X$ belongs to $\mathcal{G}_0[-\mathrm{sup}X]\ast\mathcal{P}[-\mathrm{sup}X+1]\ast\cdots\ast\mathcal{P}[-\mathrm{sup}X+n]$.
\end{enumerate}
\item[(2)] If $0\not\simeq X\in{\mathrm{D}^{\mathrm{b}}_{\mathrm{f}}(A)}$ with $\mathrm{amp}X<\mathrm{amp}A$ and $n$ is a natural number, then the following are equivalent:
\begin{enumerate}
\item[(i)] $\mathrm{G}\textrm{-}\mathrm{dim}_AX\leq n+\mathrm{inf}A-\mathrm{inf}X$;

\item[(ii)] $X\oplus X[\mathrm{amp}A-\mathrm{amp}X]$ belongs to $\mathcal{P}[-\mathrm{sup}X]\ast\cdots\ast\mathcal{P}[-\mathrm{sup}X+n-1]\ast\mathcal{G}_0[-\mathrm{sup}X+n]$.
   \end{enumerate}
 \item[(3)]  The following conditions are equivalent:
   \begin{enumerate}
\item[(i)] $A$ is local Gorenstein;

\item[(ii)] {\rm G}-$\mathrm{dim}_A\bar{k}<\infty$;

\item[(iii)] {\rm G}-$\mathrm{dim}_A{X}<\infty$ for any $X\in{\mathrm{D}^{\mathrm{b}}_{\mathrm{f}}(A)}$.
   \end{enumerate}
\end{enumerate}
\end{thm}

A few comments on Theorem \ref{thm:main} are in order. First, Theorem \ref{thm:main} provides a criteria for a natural number to be an upper bound of the G-dimension of a DG-module over a commutative noetherian local DG-ring. One major difference from the case of rings is that $\mathrm{G}\textrm{-}\mathrm{dim}_AX\geq-\mathrm{sup}X$ need not hold in the DG-setting (see Example \ref{Ex1}), it leads to the proof of two key results, Lemmas \ref{lem:3.2'} and \ref{lem:6.00}, for obtaining such a characterization is rather different with the ordinary commutative rings.

Second, it should be pointed out that the inequality $\mathrm{amp}X<\mathrm{amp}A$ in Theorem \ref{thm:main}(2) is very often met. For instance, assume that $(A,\bar{\mathfrak{m}},\bar{k})$ is a commutative noetherian local DG-ring with $0<\mathrm{amp}A<\infty$. If we choose $X=\bar{k}$ then $\mathrm{amp}X=0<\mathrm{amp}A$, as desired.

Finally, if $(A,\bar{\mathfrak{m}},\bar{k})$ is a local Gorenstein DG-ring with $0<\mathrm{amp}A<\infty$, then $\bar{k}$ has finite G-dimension by Theorem \ref{thm:main}(3), {but by \cite[Theorem 0.2]{J}, $\bar{k}$ never has finite projective dimension introduced by Bird, Shaul, Sridhar and Williamson in \cite{BSSW}}. For more explanations, we refer to Example \ref{exm:3.1}. In combination with Proposition \ref{lem:3.2}, we conclude that G-dmension is a finer invariant than projective dimension for DG-modules. We refer to Corollaries \ref{corollary:1.1}, \ref{the3.4} and \ref{cor1.6} for more evidence about this.

Let $A$ be an Artin algebra.  The \emph{little finitistic dimension}
of $A$, ${\rm fpd}{A}$, is the supremum of projective dimensions of
finitely generated $A$-modules of finite projective dimension. It is
conjectured that ${\rm fpd}{A}<\infty$ holds for Artin algebras $A$; see
Bass \cite{Bass} and \cite[Conjectures]{ARS}. This is the Finitistic Dimension Conjecture.  Recently, this invariant has been generalized by Bird, Shaul, Sridhar and Williamson in \cite{BSSW} to the setting of non-positive commutative noetherian DG-rings with bounded cohomology. More precisely, they considered the number
$$\mathrm{fpd}A=\mathrm{sup}\{\mathrm{projdim}_AX+\inf{X}\mid
X\in\mathrm{D}^{\mathrm{b}}_{\mathrm{f}}(A)\ \textrm{with}\ \mathrm{projdim}_AX<\infty\}.$$

It should be noted that the little finitistic dimensions of an
algebra can alternatively be computed by G-dimension; see for instance \cite[Lemma 4.4]{Xi}. As the first application of Theorem \ref{thm:main}, we have the following corollary which shows that the little finitistic dimension $\mathrm{fpd}(A)$ over a DG-ring $A$ can be computed by G-dimension defined in Definition \ref{def:1.2} above.

\begin{cor}\label{corollary:1.1} If $A$ is a commutative noetherian local DG-ring with $\mathrm{amp}A<\infty$, then we have $$\mathrm{fpd}A=\mathrm{sup}\{\mathrm{G}\textrm{-}\mathrm{dim}_AX+\inf X\mid
X\in\mathrm{D}^{\mathrm{b}}_{\mathrm{f}}(A)\ {\rm with}\ \mathrm{G}\textrm{-}\mathrm{dim}_AX<\infty\}.$$
\end{cor}
As a direct consequence of Theorem \ref{thm:main}(3) and Corollary \ref{corollary:1.1}, it follows that a local Gorenstein DG-ring $A$ has finite little finitistic dimension, {which gives a new proof of a very particular case of \cite[Theorem A]{BSSW}}.

Let $A$ be a commutative noetherian local ring (not a DG-ring). A close relation between the class of maximal Cohen-Macaulay modules and the class of modules of G-dimension zero over $A$ can be shown in the following fact; see for instance \cite[3.3]{CPST}.

\begin{fact}\label{fact:1.1} Let $A$ be a commutative noetherian local ring. Denote by $\mathcal{M}$ the class of maximal Cohen-Macaulay modules over $A$. Then the following hold:
\begin{enumerate}
\item $A$ is Cohen-Macaulay if and only if $\mathcal{G}\subseteq \mathcal{M}$.

\item $A$ is Gorenstein if and only if $\mathcal{G}=\mathcal{M}$.
\end{enumerate}
\end{fact}

Recently, the theory of Cohen-Macaulay rings and Cohen-Macaulay modules has been extended by Shaul \cite{Shaul} to the setting of commutative noetherian DG-rings. There are many examples of local Cohen-Macaulay DG-rings, in particular local Gorenstein
DG-rings are Cohen-Macaulay. As the second application of Theorem \ref{thm:main}, we can generalize the above fact to the setting of commutative noetherian local DG-rings.

\begin{cor}\label{the3.4} Let $(A,\bar{\mathfrak{m}})$ be a commutative noetherian local DG-ring with $\mathrm{amp}A<\infty$. Denote by $\mathcal{M}$ the class of maximal local-Cohen-Macaulay DG-modules over $A$. Set $$\mathcal{H}=\{X\in\mathcal{G}\mid
\mathrm{ampR}\Gamma_{\bar{\mathfrak{m}}}X\geq\mathrm{amp}X=\mathrm{amp}A\}.$$  Then the following hold:
\begin{enumerate}
\item $A$ is local-Cohen-Macaulay if and only if $\mathcal{H}\subseteq\mathcal{M}$.

\item $A$ is Gorentein if and only $\mathcal{H}=\mathcal{M}$.
\end{enumerate}
\end{cor}

In the study of stable homological algebra and Tate cohomology for two sided noetherian rings $A$, Buchweitz \cite{Buchweitz} introduced the Verdier quotient $\textrm{D}_{\mathrm{sg}}(A):= \textrm{D}^{\mathrm{b}}(A)/\textrm{K}^{\mathrm{b}}(A)$, where $\textrm{K}^{\mathrm{b}}(A)$ is the bounded homotopy category of finitely generated projective left $A$-modules. Later on, this category was reconsidered by Orlov \cite{Orlov}. Since it measures the homological singularity of the category of finitely generated left $A$-modules in sense that $\textrm{D}_{\mathrm{sg}}(A)=0$ if and only if the global dimension of $A$ is finite, we call it the {\it singularity category}. {Buchweitz-Happel Theorem} (\cite[Theorem 4.4.1]{Buchweitz}, see also \cite[Theorem 4.6]{H2}) says that there is a fully faithful triangle functor $F:\underline{\mathcal{G}}\to\textrm{D}_{sg}(A)$ provided that $A$ is a commutative local Gorenstein ring. Also, it was shown in \cite{BJO} that the inverse of this theorem holds true (see also \cite{Be1,KZ}). Recently, Buchweitz-Happel Theorem has been generalized by Jin to a proper noncommutative Gorenstein DG-algebra over a field $k$ (see \cite[Theorem 0.3 and Assumption 0.1]{jin20}).

As the third application of Theorem \ref{thm:main}, we extend the classical Buchweitz-Happel Theorem and its inverse from commutative noetherian local rings to the setting of commutative noetherian local DG-rings. For notation and notions, we refer the reader to Section \ref{subsection:4.3}.

\begin{cor}\label{cor1.6} Assume that $A$ is a commutative noetherian local DG-ring with $\mathrm{amp}A<\infty$. Let $\mathcal{A} = \{X \in \mathcal{R}(A)\hspace{0.03cm}~|\hspace{0.03cm}\sup{X}\leq0\text{~~ and~~}{\rm G}\text{-}\mathrm{dim}_AX\leq0\}$, and let ${\rm D}_{sg}(A):=\mathrm{D}_{\mathrm{f}}^{\mathrm{b}}(A)/ \langle\mathcal{P}\rangle$ be the singularity category over $A$. Then
 the functor $$F:\underline{\mathcal{A}}\rightarrow {\rm D}_{sg}(A)$$ is a triangle equivalence if and only if $A$ is a local Gorenstein DG-ring.
\end{cor}

Remark that if $A$ is a commutative noetherian local DG-algebra over a field $k$, a Gorenstein ring in our setting is a proper Gorenstein DG-algebra defined by Jin in \cite[Assumption 0.1]{jin20}. Therefore, over a commutative noetherian local non-positive DG algebra over a field $k$, Corollary \ref{cor1.6} conveys that the inverse of Buchweitz-Happel Theorem obtained by Jin \cite[Theorem 0.3]{jin20} is also true.

The article is organized as follows. In Section \ref{pre}, we introduce notation, definitions and
basic facts needed for proofs. Section \ref{proof} is devoted to giving the proof of Theorem \ref{thm:main}. In Section 4, we apply Theorem \ref{thm:main} to establish the connection between G-dimensions and the little finitistic dimensions of $A$, to obtain a DG version of Fact \ref{fact:1.1}, and to show a DG-version of Buchweitz-Happel Theorem and its inverse over $A$, and therefore prove the three corollaries.

\section{Preliminaries}\label{pre}
An \emph{associative DG-ring} $A$ is a $\mathbb{Z}$-graded ring
$A=\bigoplus_{i\in\mathbb{Z}}A^i$
equipped with a $\mathbb{Z}$-linear map $d:A\rightarrow A$ of degree +1 such that $d\circ d=0$, and such that
\begin{equation}\label{eq:2.1}d(a\cdot b)=d(a)\cdot b+(-1)^{i}a\cdot d(b)
\end{equation}
for all $a\in A^i$ and $b\in A^j$. A DG-ring $A$ is
called \emph{commutative} if for all $a\in A^i$ and $b\in A^j$, one has $b\cdot a=(-1)^{i\cdot j}a\cdot b$, and $a^2=0$
if $i$ is odd.
A DG-ring $A$ is called \emph{non-positive} if $A^i=0$ for all $i>0$.
 A non-positive DG-ring $A$ is called \emph{noetherian} if the
ring $\mathrm{H}^0(A)$ is noetherian and the $\mathrm{H}^0(A)$-modules $\mathrm{H}^i(A)$ are
finitely generated for all $i<0$.  \textbf{All DG-rings in this paper will be assumed
to be commutative and non-positive.}

A DG-module $M$ over $A$ is a graded $A$-module together with a differential of degree +1 satisfying the
Leibniz rule similar to \eqref{eq:2.1}. The DG-$A$-modules form an abelian category, and by inverting the quasi-isomorphisms we obtain the derived category $\mathrm{D}(A)$ which is triangulated. We refer the reader to \cite{ye20} for more details about DG-rings and their derived categories.

If $A$ is a noetherian DG-ring, we say that $M \in\mathrm{D}(A)$ has finitely generated
cohomology if for all $n \in \mathbb{Z}$, the $\mathrm{H}^{0}(A)$-modules $\mathrm{H}^{n}(M)$ are finitely generated. We
denote by $\mathrm{D}_{\mathrm{f}}(A)$ the full triangulated subcategory of $\mathrm{D}(A)$ consisting of DG-modules
with finitely generated cohomology. We also set $\mathrm{D}^{-}_{\mathrm{f}}(A) = \mathrm{D}_{\rm f}(A)\cap \mathrm{D}^{-}(A)$. Similarly we will consider $\mathrm{D}^{+}_{\mathrm{f}}(A)$ and $\mathrm{D}^{\mathrm{b}}_{\mathrm{f}}(A)$. All these are full triangulated subcategories of $\mathrm{D}(A)$.
If $A$ is a noetherian DG-ring, and if the noetherian ring $\mathrm{H}^{0}(A)$ is local
with maximal ideal $\bar{\mm}$ and residue field $\bar{k}$, we will say that $(A,\bar{\mm})$ (or $(A,\bar{\mm},\bar{k})$) is a noetherian local DG-ring.

We recall the definition of the projective and injective dimensions of DG-modules introduced by Bird, Shaul, Sridhar and Williamson.

\begin{definition}\label{df:2.1} $($\cite[Definition 2.1]{BSSW}$)$ Let $A$ be a DG-ring and $M \in \mathrm{D}(A)$.
\begin{enumerate}
\item The \emph{projective dimension} of $M$ is defined by
$$\mathrm{projdim}_AM=\mathrm{inf}\{n\in\mathbb{Z}\mid\mathrm{Ext}^i_A(M,N)=0\ \textrm{for\ any}\ N\in\mathrm{D}^\mathrm{b}(A)\ \textrm{and\ any}\ i>n+\mathrm{sup}N\}.$$

\item The \emph{injective dimension} of $M$ is defined by
$$\mathrm{injdim}_AM=\mathrm{inf}\{n\in\mathbb{Z}\mid\mathrm{Ext}^i_A(N,M)=0\ \textrm{for\ any}\ N\in\mathrm{D}^\mathrm{b}(A)\ \textrm{and\ any}\ i>n-\mathrm{inf}N\}.$$
\end{enumerate}
\end{definition}

Following \cite{s19}, let us
recall the notion of local cohomology and local homology functors over commutative DG-rings. Let $A$ be a commutative DG-ring, and let $\bar{\mathfrak{a}}\subseteq \mathrm{H}^{0}(A)$ be a finitely generated ideal. The category of derived $\bar{\mathfrak{a}}$-torsion DG-modules over $A$, denoted by
$\mathrm{D}_{\bar{\mathfrak{a}}\text{-}\mathrm{tor}}(A)$, is the full triangulated subcategory of $\mathrm{D}(A)$ consisting of DG-modules $M$ such that the $\mathrm{H}^{0}(A)$-modules $\mathrm{H}^{n}(M)$ are $\bar{\mathfrak{a}}$-torsion for all $n\in \mathbb{Z}$. One can show that the inclusion functor
$F : \mathrm{D}_{\bar{\mathfrak{a}}\text{-}\mathrm{tor}}(A)\rightarrow \mathrm{D}(A)$
has a right adjoint
$G : \mathrm{D}(A) \rightarrow \mathrm{D}_{\bar{\mathfrak{a}}\text{-}\mathrm{tor}}(A)$,
and composing this right adjoint with the inclusion, one obtains a triangulated functor
$$\mathrm{R}\Gamma_{\bar{\mathfrak{a}}} : \mathrm{D}(A) \rightarrow \mathrm{D}(A),$$
which we call the \emph{derived torsion} or \emph{local cohomology functor} of $A$ with respect to
$\bar{\mathfrak{a}}$. The functor $\mathrm{R}\Gamma_{\bar{\mathfrak{a}}}$ has a right adjoint which is denoted by
$$\mathrm{L}\Lambda_{\bar{\mathfrak{a}}} : \mathrm{D}(A) \rightarrow \mathrm{D}(A),$$
This functor is called the \emph{local homology} or \emph{derived completion functor} with respect to
$\bar{\mathfrak{a}}$.

We recall the following definitions introduced by Shaul in \cite{Shaul}.

\begin{definition}Let $(A,\bar{\mm})$  be a noetherian local DG-ring.

\begin{enumerate}
\item We define the \emph{local cohomology Krull dimension}
of $M\in\mathrm{D}^{-}(A)$ to be
$$\mathrm{lc.dim}_AM:= \sup_{l\in \mathbb{Z}}\{\mathrm{dim}(\mathrm{H}^{l}(M)) + l\}.$$

\item We define the \emph{depth} of $N\in\mathrm{D}^{+}(A)$ to be the number
$$\text{\rm depth}_{A}N:= \inf\mathrm{RHom}_{A}(\bar{k},N).$$
\end{enumerate}
\end{definition}

\begin{definition} \label{def} Let $(A,\bar{\mm})$  be a noetherian local DG-ring with $\mathrm{amp}A<\infty$.

\begin{enumerate}
\item  $A$ is called \emph{local Cohen-Macaulay} if $\mathrm{amp}\mathrm{R}\Gamma_{\bar{\mm}}(A) = \mathrm{amp}A$.

\item We say that $M \in \mathrm{D}^{\mathrm{b}}_{\mathrm{f}}(A)$ is a
\emph{local-Cohen-Macaulay DG-module} if there are equalities
$$\mathrm{amp}M = \mathrm{amp}A = \mathrm{amp}\mathrm{R}\Gamma_{\bar{\mm}}(M).$$

\item A local-Cohen-Macaulay DG-module $M$ is \emph{maximal local-Cohen-Macaulay} if
$$\mathrm{lc.dim}_AM= \sup M + \mathrm{dim}\mathrm{H}^{0}(A).$$
 We denote by $\mathcal{M}$ the full subcategory of maximal local-Cohen-Macaulay DG-modules.
\end{enumerate}
\end{definition}

\textbf{From now until the end of the paper, we always assume that $(A,\bar{\mathfrak{m}},\bar{k})$ is a commutative noetherian local DG-ring with $\mathrm{amp}A<\infty$.}

\section{Proof of Theorem \ref{thm:main}}\label{proof}
This section is devoted to a proof of the statements of Theorem \ref{thm:main}. We start with
the following proposition.

\begin{prop}\label{lem3.5} Let $X'\rightarrow X\rightarrow X''\rightarrow X'[1]$ be an exact triangle in
$\mathrm{D}^{\mathrm{b}}_{\mathrm{f}}(A)$. If two of the DG-modules belong to $\mathcal{R}(A)$, then so is the third.
\end{prop}
\begin{proof} If two of $\mathrm{RHom}_{A}(X',A),\mathrm{RHom}_{A}(X,A)$ and $\mathrm{RHom}_{A}(X'',A)$ are in $\mathrm{D}^{\mathrm{b}}_{\mathrm{f}}(A)$,
then so is the third. Using the abbreviated notation $((-, A),A):= \mathrm{RHom}_{A}(\mathrm{RHom}_{A}(-,A),A)$, we have
the following commutative diagram of {exact triangles} in $\mathrm{D}^{\mathrm{b}}_{\mathrm{f}}(A)$:
$$\xymatrix@C=20pt@R=18pt{
  X' \ar[d]_{\delta^{X'}} \ar[r]^{} & X \ar[d]_{\delta^{X}} \ar[r]^{} & X'' \ar[d]_{\delta^{X''}} \ar[r]^{} & X'[1] \ar[d]_{\delta^{X'}[1]} \\
  ((X', A),A) \ar[r]^{} &((X, A),A) \ar[r]^{} & ((X'', A),A) \ar[r]^{} & ((X', A),A)[1].}$$
Therefore, if two of the morphisms $\delta^{X'},\delta^{X}$ and $\delta^{X''}$ are isomorphisms,
then so is the third. The rest follows from the fact that $\mathrm{D}^{\mathrm{b}}_{\mathrm{f}}(A)$ is the full triangulated subcategory of $\mathrm{D}(A)$.
\end{proof}

{The following lemma was proven in a greater generality by P. J${\o}$rgensen in \cite{J}.}

\begin{lem}\label{lem0.2}{\rm (\cite[Theorem 0.2]{J})} Let $0\not\simeq X\in\mathrm{D}^{\mathrm{b}}_{\mathrm{f}}(A)$ with $\mathrm{projdim}_AX<\infty$. Then $\mathrm{amp}X\geq\mathrm{amp}A$.
\end{lem}

\begin{lem}\label{lem0.3} If $X\in\mathcal{R}(A)$ with $\mathrm{amp}X\geq\mathrm{amp}A$, then $\mathrm{G}\textrm{-}\mathrm{dim}_AX\geq-\mathrm{sup}X$.
\end{lem}
\begin{proof} By assumption, one has the following (in)equalities
\begin{center}$\begin{aligned}\mathrm{sup}X&\geq\mathrm{inf}X-\mathrm{inf}A\\
&=\mathrm{inf}\mathrm{RHom}_A(\mathrm{RHom}_A(X,A),A)-\mathrm{inf}A\\
&\geq\mathrm{inf}A-\mathrm{sup}\mathrm{RHom}_A(X,A)-\mathrm{inf}A\\
&=-\mathrm{sup}\mathrm{RHom}_A(X,A),\end{aligned}$\end{center}
where {the first inequality is by assumption and the second one is by \cite[Lemma 3.2(i)]{BSSW}}, which implies that $\mathrm{G}\textrm{-}\mathrm{dim}_AX=\mathrm{sup}\mathrm{RHom}_A(X,A)\geq-\mathrm{sup}X$.
\end{proof}

We thank Dong Yang for communicating to us the next example which implies that $\mathrm{G}\textrm{-}\mathrm{dim}_AX\geq-\mathrm{sup}X$ need not hold in the DG-setting in general.

\begin{Ex}\label{Ex1}{\rm Let $A=k[x]/(x^2)$ considered as a DG $k$-algebra, where $k$ is a field and deg$(x)=-1$. Let $S=A/(x)$, which concentrated in degree 0. Then $S$ has a sppj resolution
 $$\cdots\rightarrow A[2]\s{x}\rightarrow A[1]\s{x}\rightarrow A\rightarrow S\rightarrow0.$$
 Its  total complex $F$ is a minimal semi-free resolution of $S$. Now we can deduce that the DG-module $\mathrm{Hom}(F,A)$ is the total complex of  $$0\rightarrow A\s{x}\rightarrow A[-1]\s{x}\rightarrow A[-2]\rightarrow\cdots.$$ Then $\mathrm{RHom}(S,A)\simeq\mathrm{Hom}(F,A)\simeq S[1]$ and $S\in\mathcal{R}(A)$. So in this case, $\sup S=0$, but $\mathrm{G}\textrm{-}\mathrm{dim}_AS=\sup\mathrm{RHom}(S,A)=-1$.}
\end{Ex}
\begin{definition}\label{def:1.2'} A DG-module $X\in\mathcal{R}(A)$ is said to be in the G-class $\mathcal{G}$ if either $\mathrm{G}\textrm{-}\mathrm{dim}_AX=-\mathrm{sup}X$, or $X=0$, and we denote
 by $\mathcal{G}_0$ the full subcategory of $\mathcal{G}$ consisting of objects $G$ such that either $\mathrm{amp}G\geq\mathrm{amp}A$ and $\mathrm{G}\textrm{-}\mathrm{dim}_AG=-\mathrm{sup}G=0$, or $G=0$.
\end{definition}


\begin{rem}\label{lem:3.1}{\rm (1) Let $X\rightarrow Y\rightarrow Z\rightsquigarrow$ be an exact triangle in $\mathcal{R}(A)$. The exact triangle $\mathrm{RHom}_A(Z,A)\rightarrow \mathrm{RHom}_A(Y,A)\rightarrow \mathrm{RHom}_A(X,A)\rightsquigarrow$ yields that
\begin{center}$\mathrm{G}\textrm{-}\mathrm{dim}_AY\leq\mathrm{max}\{\mathrm{G}\textrm{-}\mathrm{dim}_AX,
\mathrm{G}\textrm{-}\mathrm{dim}_AZ\}$.\end{center}

(2) For $0\not\simeq X\in\mathcal{R}(A)$ with $\mathrm{amp}X\geq\mathrm{amp}A$, it follows from Lemma \ref{lem0.3} that $\mathrm{G}\textrm{-}\mathrm{dim}_AX=-\mathrm{sup}X$ if and only if $X\in\mathcal{G}_0[-\mathrm{sup}X]$.

(3) For any $i\in\mathbb{Z}$, one has $\mathcal{P}[i]\subseteq \mathcal{G}_0[i]$ by \cite[Lemma 2.14]{Mi18} and Lemma \ref{lem0.2}.

(4) For any $X\in\mathcal{R}(A)$ and $n\in\mathbb{Z}$, one has $\mathrm{G}\textrm{-}\mathrm{dim}_AX[n]=\mathrm{G}\textrm{-}\mathrm{dim}_AX+n$.

(5) {For any $X\in\mathrm{D}^\mathrm{b}(A)$}, there is a sppj $f:P\rightarrow X$ and an exact triangle $Y\rightarrow P\rightarrow X\rightsquigarrow$. If $\mathrm{amp}Y<\mathrm{amp}A$, then $(f,0):P\oplus P\rightarrow X$ is also a sppj morphism such that the triangle $Y\oplus P\rightarrow P\oplus P\rightarrow X\rightsquigarrow$ is exact and $\mathrm{amp}(Y\oplus P)\geq\mathrm{amp}A$ by Lemma \ref{lem0.2}.}
\end{rem}

We give several useful properties about sppj morphisms.

\begin{lem}\label{lem:3.01} {\it{Let $0\not\simeq X\in\mathcal{R}(A)$ with $\mathrm{amp}X\geq\mathrm{amp}A$ and $f:P\rightarrow X$ a sppj morphism with $Y=\mathrm{cn}(f)[-1]$. If $\mathrm{G}\textrm{-}\mathrm{dim}_AX+\mathrm{sup}X\geq1$, then
$\mathrm{G}\textrm{-}\mathrm{dim}_AY=\mathrm{G}\textrm{-}\mathrm{dim}_AX-1$.}}
\end{lem}
\begin{proof} By Proposition \ref{lem3.5}, $Y\in\mathcal{R}(A)$. Set $n=\mathrm{G}\textrm{-}\mathrm{dim}_AX+\mathrm{sup}X$ and $g:Y\rightarrow P$ be the canonical morphism. We
may assume that $\mathrm{sup}X=0$ by shifting the degree. Then $n\geq1$.

Firstly, we need to prove $\mathrm{H}^i(\mathrm{RHom}_A(Y,A))=0$ for all
$i>n-1$.
By assumption, we have  $\mathrm{H}^i(\mathrm{RHom}_A(X,A))=0$ for all
$i>n$ and
 $\mathrm{H}^i(\mathrm{RHom}_A(P,A))=0$ for all
$i>0$. Then the following exact sequence
\begin{center} $\mathrm{H}^i(\mathrm{RHom}_A(P,A))\rightarrow\mathrm{H}^i(\mathrm{RHom}_A(Y,A))\rightarrow
\mathrm{H}^{i+1}(\mathrm{RHom}_A(X,A))\rightarrow \mathrm{H}^{i+1}(\mathrm{RHom}_A(P,A))$\end{center} yields that
$\mathrm{H}^i(\mathrm{RHom}_A(Y,A))=0$ for all
$i>n-1$.
Next, since $\mathrm{H}^{n}(\mathrm{RHom}_A(X,A))\neq0$, the exact sequence $\mathrm{H}^{n-1}(\mathrm{RHom}_A(Y,A))\rightarrow\mathrm{H}^{n}(\mathrm{RHom}_A(X,A))\rightarrow 0$ implies that
$\mathrm{H}^{n-1}(\mathrm{RHom}_A(Y,A))\neq0$. Therefore, one has $\mathrm{G}\textrm{-}\mathrm{dim}_AY=n-1$.
\end{proof}

\begin{lem}\label{lem:3.2'}{\it{Let $0\not\simeq X\in\mathcal{G}$ with $\mathrm{amp}X\geq\mathrm{amp}A$. Then there exists a sppj morphism $f:P\rightarrow X$ such that the cocone $Y=\mathrm{cn}(f)[-1]$ is in $\mathcal{G}_0[-\mathrm{sup}X]$.}}
\end{lem}
\begin{proof} Let $f:P\rightarrow X$ be a sppj morphism with $P\in\mathcal{P}[-\mathrm{sup}X]$. One has an exact triangle $Y\rightarrow P\rightarrow X\rightsquigarrow$. By Remark \ref{lem:3.1}(5), we may assume that $\mathrm{amp}Y\geq\mathrm{amp}A$. Then $\mathrm{G}\textrm{-}\mathrm{dim}_AP=-\mathrm{sup}X$ by \cite[Lemma 2.14]{Mi18} and \cite[Proposition 4.4(3)]{ya20}, so $\mathrm{sup}\mathrm{RHom}_A(Y,A)\leq-\mathrm{sup}X\leq-\mathrm{sup}Y$ since $\mathrm{sup}Y\leq \mathrm{sup}X$.
Hence
Remark \ref{lem:3.1}(2) implies that $Y\in\mathcal{G}_0[-\mathrm{sup}X]$.
\end{proof}

\begin{lem}\label{lem:6.00} {\it{Let $0\not\simeq X\in\mathcal{G}$ with $\mathrm{amp}X\geq\mathrm{amp}A$. Then there exists a morphism $f:X\rightarrow P$ such that $Y=\mathrm{cn}(f)\in\mathcal{G}_0[-\mathrm{sup}X]$.}}
\end{lem}
\begin{proof} Since $\mathrm{sup}\mathrm{RHom}_A(X,A)=-\mathrm{sup}X$ and $\mathrm{inf}\mathrm{RHom}_A(X,A)\geq\mathrm{inf}A-\mathrm{sup}X$, it follows that $\mathrm{amp}\mathrm{RHom}_A(X,A)\leq\mathrm{amp}A$. Set $n=\mathrm{amp}A-\mathrm{amp}\mathrm{RHom}_A(X,A)$ and $X'=\mathrm{RHom}_A(X,A)\oplus\mathrm{RHom}_A(X,A)[n]$. Then $\mathrm{amp}X'=\mathrm{amp}A$ and $\mathrm{sup}X'=\mathrm{sup}\mathrm{RHom}_A(X,A)$. By Lemma \ref{lem:3.2'}, there exists a sppj morphism $g:P\rightarrow X'$ such that $P\in\mathcal{P}[-\mathrm{sup}X']$, so we have a morphism \begin{center}$g^\ast:X\oplus X[-n]\simeq\mathrm{RHom}_{A}(X',A)\rightarrow\mathrm{RHom}_{A}(P,A)$\end{center} and an exact triangle $X\stackrel{f}\rightarrow \mathrm{RHom}_{A}(P,A)\rightarrow Y\rightsquigarrow$, where $f$ is the composition of $X\rightarrow X\oplus X[-n]$ and $g^\ast$.
Since $\mathrm{RHom}_{A}(P,A)\in\mathcal{P}[-\mathrm{sup}X]$, $Y\in\mathcal{R}(A)$.  By Remark \ref{lem:3.1}(5), we may assume that $\mathrm{amp}Y\geq\mathrm{amp}A$.
We need to show that $Y\in\mathcal{G}_0[-\mathrm{sup}X]$. Note that $\mathrm{H}^{-\mathrm{sup}X}(\mathrm{RHom}(f,A))$ is surjective, $\mathrm{G}\textrm{-}\mathrm{dim}_AY=\mathrm{sup}\mathrm{RHom}_{A}(Y,A)\leq\mathrm{sup}\mathrm{RHom}_{A}(X,A)=
\mathrm{G}\textrm{-}\mathrm{dim}_AX
=-\mathrm{sup}X$. If $\mathrm{G}\textrm{-}\mathrm{dim}_A\mathrm{Y}+\mathrm{sup}Y\geq1$, then $\mathrm{sup}X\leq\mathrm{sup}\mathrm{Y}-1$, and so $\mathrm{H}^{\mathrm{sup}\mathrm{Y}}(Y)=0$, this is a contradiction. Therefore, $Y\in\mathcal{G}_0[-\mathrm{sup}X]$.
\end{proof}

With the above preparations, now we prove Theorem \ref{thm:main}.

{\bf Proof of Theorem \ref{thm:main}.} (1) Let $X$ be a DG-module in $\mathrm{D}^{\mathrm{b}}_{\mathrm{f}}(A)$ with $\mathrm{amp}X\geq\mathrm{amp}A$ and $n$ a natural number. Then $(ii) \Longrightarrow (iii)$ follows from Lemma \ref{lem:3.2'} as $e-\mathrm{sup}X_e\leq n-\mathrm{sup}X$. $(iii) \Longrightarrow(i)$ and $(iv) \Longrightarrow(i)$ hold by the definition.

To show $(i) \Longrightarrow (ii)$, we assume that $P_{\bullet}$ is a sppj resolution of $X$ and $X_{i+1}\stackrel{g_{i+1}}\rightarrow P_i\stackrel{f_i}\rightarrow X_i\rightsquigarrow$ are the corresponding exact triangles
such that $f_i$ is a sppj morphism and $\mathrm{amp}X_i\geq\mathrm{amp}A$ for $i\geq 0$ with $X_0:=X$. Fix $i\geq 1$. If  $\mathrm{G}\textrm{-}\mathrm{dim}_AX_{j-1}+\mathrm{sup}X_{j-1}>0$ for $1\leq j\leq i$, then $\mathrm{G}\textrm{-}\mathrm{dim}_AX_{j}=\mathrm{G}\textrm{-}\mathrm{dim}_AX_{j-1}-1$ for $1\leq j\leq i$ by Lemma \ref{lem:3.01}, which implies
\begin{center}$\mathrm{G}\textrm{-}\mathrm{dim}_AX_{i}+i=\mathrm{G}\textrm{-}\mathrm{dim}_AX$.\end{center} Since $\mathrm{sup}X_i\leq \mathrm{sup}X$, it follows that the set
$\{i\geq 1\mid\mathrm{G}\textrm{-}\mathrm{dim}_AX_{i-1}+\mathrm{sup}X_{i-1}>0\}$ is finite. Set $e:=\mathrm{max}\{i\geq 1\mid\mathrm{G}\textrm{-}\mathrm{dim}_AX_{i-1}+\mathrm{sup}X_{i-1}>0\}$. Then $\mathrm{G}\textrm{-}\mathrm{dim}_AX_{e}+e=e-\mathrm{sup}X_e=\mathrm{G}\textrm{-}\mathrm{dim}_AX\leq n-\mathrm{sup}X$
and $X_e\in\mathcal{G}_0[-\mathrm{sup}X_e]$.

$(i) \Longrightarrow (iv)$ If $\mathrm{G}\textrm{-}\mathrm{dim}_AX=-\mathrm{sup}X$, then $X\in\mathcal{G}_0[-\mathrm{sup}X]$ by Remark \ref{lem:3.1}(2). Assume that $\mathrm{G}\textrm{-}\mathrm{dim}_AX>-\mathrm{sup}X$. By analogy with the proof of Lemma \ref{lem:6.00},
one can obtain an exact triangle $Y\rightarrow X\rightarrow P^\ast\rightsquigarrow$, which induces the exact triangle
\begin{equation}\label{equ:3.2}
\mathrm{RHom}_{A}(Y,A)[-1]\rightarrow P\stackrel{f}\rightarrow\mathrm{RHom}_{A}(X,A)\rightarrow\mathrm{RHom}_{A}(Y,A),
\end{equation}where $P^\ast=\mathrm{RHom}_{A}(P,A)\in\mathcal{P}[\mathrm{sup}\mathrm{RHom}_{A}(X,A)]$. By Remark \ref{lem:3.1}(5), we may assume that $\mathrm{amp}Y\geq\mathrm{amp}A$.
Since $\mathrm{sup}P^\ast=-\mathrm{sup}\mathrm{RHom}_{A}(X,A)=\mathrm{sup}X-n<\mathrm{sup}X$, one has $\mathrm{sup}Y=\mathrm{sup}X$. As $\mathrm{H}^{\mathrm{sup}\mathrm{RHom}_{A}(X,A)}(f)$ is surjective, it follows from the triangle (\ref{equ:3.2}) that \begin{center}$\mathrm{G}\textrm{-}\mathrm{dim}_AY=\mathrm{sup}\mathrm{RHom}_{A}(Y,A)
\leq\mathrm{sup}\mathrm{RHom}_{A}(X,A)-1=\mathrm{G}\textrm{-}\mathrm{dim}_AX-1=n-1-\mathrm{sup}Y$.\end{center}
Now by induction, one can obtain $X\in\mathcal{G}_0[-\mathrm{sup}X]\ast\mathcal{P}[-\mathrm{sup}X+1]\ast\cdots\ast\mathcal{P}[-\mathrm{sup}X+n]$.

(2) Set $\bar{X}=X\oplus X[\mathrm{amp}A-\mathrm{amp}X]$. Then $\bar{X}\in\mathrm{D}^{\mathrm{b}}_{\mathrm{f}}(A)$ with $\mathrm{amp}\bar{X}=\mathrm{amp}A$ and $\mathrm{sup}\bar{X}=\mathrm{sup}X$. As $\textrm{G-dim}_AX+(\mathrm{amp}A-\mathrm{amp}X)=\textrm{G-dim}_A\bar{X}$, the statement follows from (1).

(3) Evidently, $(iii)$ is stronger than $(ii)$, so it is sufficient to prove that $(i)$ implies $(iii)$ and $(ii)$ implies $(i)$.

$(i)\Rightarrow(iii)$ As $A$ is local Gorenstein, $\mathrm{injdim}_AA<\infty$ and $\mathrm{sup}\mathrm{RHom}_{A}(X,A)\leq\mathrm{injdim}_AA-\mathrm{inf}X$, it follows that $\mathrm{RHom}_{A}(X,A)\in\mathrm{D}^{\mathrm{b}}_{\mathrm{f}}(A)$ and $X\in\mathcal{R}(A)$ for any $X\in\mathrm{D}^{\mathrm{b}}_{\mathrm{f}}(A)$. Consequently, for
any $X\in\mathrm{D}^{\mathrm{b}}_{\mathrm{f}}(A)$, we have $\textrm{G-dim}_AX+\mathrm{inf}X\leq\mathrm{injdim}_AA$, as desired.

$(ii)\Rightarrow(i)$
Assume that G-dim$_A\bar{k} < \infty$. Then $\mathrm{sup}\mathrm{RHom}_{A}(\bar{k},A)<\infty$. Hence \cite[Proposition 4.2 and Theorem 4.7]{ya20} implies that
$\mathrm{injdim}_AA<\infty$ and $A$ is local Gorenstein.
\hfill$\Box$

\section{Applications}
In this section, we will apply Theorem \ref{thm:main} to establish the connection between G-dimensions and the little finitistic dimensions of $A$, to examine the DG version of the Fact \ref{fact:1.1}, and to establish a DG-version of Buchwtweiz-Happel Theorem and
its inverse.
\subsection{Relations between G-dimensions and the little finitistic dimensions}
In this subsection we will apply Theorem \ref{thm:main} to show that the little finitistic dimension $\mathrm{fpd}(A)$ can be computed by G-dimension of DG-modules in $\mathrm{D}^{\mathrm{b}}_{\mathrm{f}}(A)$.

Yekutieli \cite{ye16} introduced a DG-version of projective dimension, $\mathrm{pd}M$, for a DG-module $M$
over a DG-algebra which are different from the definition ${\rm projdim}_AM$ in Definition \ref{df:2.1}. By \cite[Theorems 2.22 and 3.21]{Mi18}, one has
${\rm projdim}_AM=\mathrm{pd}M-\mathrm{sup}M$.

\begin{prop}\label{lem:3.2}{\it{Let $0\not\simeq X\in\mathrm{D}^{\mathrm{b}}_{\mathrm{f}}(A)$. One has \begin{center}$\mathrm{G}\textrm{-}\mathrm{dim}_AX\leq\mathrm{projdim}_AX$,\end{center}and equality holds if $\mathrm{projdim}_AX<\infty$.}}
\end{prop}
\begin{proof} The inequality is trivial if $X$ is of infinite projective dimension. Next, assume that  $n=\mathrm{projdim}_AX+\mathrm{sup}X<\infty$. It follows from \cite[Proposition 2.26]{Mi18} that $X\in\mathcal{P}[-\mathrm{sup}X]\ast\mathcal{P}[-\mathrm{sup}X+1]\ast\cdots\ast\mathcal{P}[-\mathrm{sup}X+n]$. Hence $\mathrm{G}\textrm{-}\mathrm{dim}_AX\leq n-\mathrm{sup}X$ by Theorem \ref{thm:main}(1) and Lemma \ref{lem0.2}. This yields that $\mathrm{G}\textrm{-}\mathrm{dim}_AX\leq\mathrm{projdim}_AX$.
For the last part, one has
\begin{center}$\mathrm{projdim}_AX=\mathrm{sup}\mathrm{RHom}_{A}(X,A)
=\mathrm{G}\textrm{-}\mathrm{dim}_AX$\end{center}by \cite[Proposition 2.2]{Mi19} and  \cite[Proposition 4.4(3)]{ya20}, as desired.
\end{proof}

We employ an example in \cite{sh21} to illustrate Proposition \ref{lem:3.2} that a DG-module with finite G-dimension may have infinite projective dimension.

\begin{ex}\label{exm:3.1} {\rm Let $(B,\mathfrak{m},k)$ be a commutative noetherian local ring and $D$ a dualizing complex over $B$ with $\mathrm{sup}D<0$. Then the trivial extension DG-ring $A:=B\ltimes D$ is a local Gorenstein DG-ring with $0<\mathrm{amp}A<\infty$ by \cite[Example 7.2]{sh21}, and hence G-$\mathrm{dim}_Ak\leq\mathrm{injdim}_AA<\infty$, but $\mathrm{projdim}_Ak=\infty$ by Lemma \ref{lem0.2}}.
\end{ex}


The next lemma is an immediate consequence of equivalence of $(i)$ and $(iv)$ in Theorem \ref{thm:main}(1).

\begin{lem}\label{lem6.1} Let $X$ be in $\mathrm{D}^{\mathrm{b}}_{\mathrm{f}}(A)$ with $\mathrm{amp}X\geq\mathrm{amp}A$. If
$\mathrm{G}\textrm{-}\mathrm{dim}_AX=n$, then there exists an exact triangle \begin{center}$G\rightarrow X
\rightarrow K\rightsquigarrow$,\end{center} such that
$G\in\mathcal{G}_0[-\mathrm{sup}X]$ and $\mathrm{projdim}_AK=n-1$. For $n=-\mathrm{sup}X$,
we understand that $K=0$.
\end{lem}

The following result conveys that G-dimensions and the little finitistic dimensions are closely related to each other, which contains Corollary \ref{corollary:1.1} in the introduction.

\begin{thm}\label{lem6.2} Let $X$ be an object in $\mathrm{D}^{\mathrm{b}}_{\mathrm{f}}(A)$  with $\mathrm{amp}X\geq\mathrm{amp}A$. If $\mathrm{G}\textrm{-}\mathrm{dim}_AX=n<\infty$, then there is a DG-module $Y\in\mathrm{D}^{\mathrm{b}}_{\mathrm{f}}(A)$ with $\mathrm{projdim}_AY=n$. Moreover, one has

\vspace{2mm}

\hspace{10mm}$\mathrm{fpd}A=\mathrm{sup}\{\mathrm{G}\textrm{-}\mathrm{dim}_AX+\mathrm{inf}X\mid
X\in\mathrm{D}^{\mathrm{b}}_{\mathrm{f}}(A)\ {\rm with}\ \mathrm{amp}X\geq\mathrm{amp}A,\ \mathrm{G}\textrm{-}\mathrm{dim}_AX<\infty\}$

\vspace{2mm}
\hspace{19.5mm}$=\mathrm{sup}\{\mathrm{G}\textrm{-}\mathrm{dim}_AX+\mathrm{inf}X\mid
X\in\mathrm{D}^{\mathrm{b}}_{\mathrm{f}}(A)\ {\rm with}\ \mathrm{G}\textrm{-}\mathrm{dim}_AX<\infty\}$.
\end{thm}
\begin{proof} We may assume that $n>-\mathrm{sup}X$. By Lemma \ref{lem6.1}, there exists an exact triangle $G
\rightarrow X\rightarrow K\rightsquigarrow$ with $\mathrm{projdim}_AK=n-1$ and $G\in\mathcal{G}_0[-\mathrm{sup}X]$, and there is an exact triangle $G\rightarrow P\rightarrow G'\rightsquigarrow$ with $P\in\mathcal{P}[-\mathrm{sup}G]$ and $G'\in\mathcal{G}_0[-\mathrm{sup}G]$ by Lemma \ref{lem:6.00}. Thus we get a commutative diagram of exact triangles in $\mathrm{D}^{\mathrm{b}}_{\mathrm{f}}(A)$:
\begin{center}$\xymatrix@C=20pt@R=18pt{
  & G'[-1] \ar[d]\ar@{=}[r]&G'[-1]\ar[d]&\\
   K[-1]\ar@{=}[d]\ar[r] & G\ar[d]\ar[r]& X\ar[d]\ar[r] &K\ar@{=}[d]\\
   K[-1]\ar[r] & P\ar[d]\ar[r]& Y \ar[d]\ar[r] &K\\
   & G'\ar@{=}[r]&G'&}$\end{center}
The exact triangle $\mathrm{RHom}_A(G',A)\rightarrow \mathrm{RHom}_A(Y,A)\rightarrow \mathrm{RHom}_A(X,A)\rightsquigarrow$ yields that $\mathrm{G}\textrm{-}\mathrm{dim}_AY=n$. Note that $\mathrm{sup}X=\mathrm{sup}Y$, so
$Y\not\in\mathcal{G}_0[-\mathrm{sup}X]$. As $\mathrm{sup}K\leq\mathrm{sup}X-1$, $P\rightarrow Y$ is a sppj morphism, it follows from the exact triangle $P\rightarrow Y\rightarrow K\rightsquigarrow$ and Proposition \ref{lem:3.2}
that $\mathrm{projdim}_AY=n$. Hence we have shown
the first statement. This yields that $\mathrm{fpd}(A)\geq\mathrm{sup}\{\mathrm{G}\textrm{-}\mathrm{dim}_AX+\mathrm{inf}X\mid
X\in\mathrm{D}^{\mathrm{b}}_{\mathrm{f}}(A)\ \textrm{with}\ \mathrm{G}\textrm{-}\mathrm{dim}_AX<\infty\}$. Also for any $Z\in\mathrm{D}^{\mathrm{b}}_{\mathrm{f}}(A)$ with $\mathrm{projdim}_AZ<\infty$, one has $\mathrm{G}\textrm{-}\mathrm{dim}_AZ\geq\mathrm{projdim}_AZ$. This shows the first equality.

If $X\in\mathrm{D}^{\mathrm{b}}_{\mathrm{f}}(A)$ with $\mathrm{amp}X<\mathrm{amp}A$, then $\mathrm{amp}(X\oplus X[\mathrm{amp}A-\mathrm{amp}X])=\mathrm{amp}A$ and $\mathrm{G}\textrm{-}\mathrm{dim}_AX=\mathrm{G}\textrm{-}\mathrm{dim}_A(X\oplus X[\mathrm{amp}A-\mathrm{amp}X])-\mathrm{amp}A+\mathrm{amp}X$, and therefore $\mathrm{G}\textrm{-}\mathrm{dim}_A(X\oplus X[\mathrm{amp}A-\mathrm{amp}X])+\mathrm{inf}(X\oplus X[\mathrm{amp}A-\mathrm{amp}X])=\mathrm{G}\textrm{-}\mathrm{dim}_AX+\mathrm{inf}X$.
Thus the second equality follows.
\end{proof}

\subsection{Comparison to maximal local-Cohen-Macaulay DG-modules}

The task of this subsection is to examine the DG version of Fact \ref{fact:1.1}.
We begin with the following AB formula for G-dimension of DG-modules.

\begin{thm}\label{lem:3.3}{\rm (AB Formula)} {\it{Let $X$ be a DG-module of finite G-dimension. One has an equality \begin{center}$\mathrm{G}\textrm{-}\mathrm{dim}_AX=\mathrm{depth}A-\mathrm{depth}_AX$.\end{center}}}
\end{thm}
\begin{proof} Since $\mathrm{G}\textrm{-}\mathrm{dim}_AX<\infty$, we have the following equalities
\begin{center}$\begin{aligned}\mathrm{depth}_AX
&=\mathrm{depth}\mathrm{RHom}_{A}(\mathrm{RHom}_{A}(X,A),A)\\
&=-\mathrm{sup}\mathrm{RHom}_{A}(X,A)+\mathrm{depth}A\\
&=-\mathrm{G}\textrm{-}\mathrm{dim}_AX+\mathrm{depth}A,
\end{aligned}$\end{center}where the second equality is by \cite[Proposition 4.9]{ya20} as $\mathrm{RHom}_{A}(X,A)\in\mathrm{D}^{\mathrm{b}}_{\mathrm{f}}(A)$, as required.
\end{proof}

Following \cite{Shaul}, if
$X\in\mathrm{D}^{\mathrm{b}}_{\mathrm{f}}(A)$ then we set
\begin{center}$\mathrm{seq.depth}_AX=\mathrm{depth}_AX-\mathrm{inf}X$,\end{center}
and call it the \emph{sequential depth} of $X$.

\begin{cor}\label{lem:3.3'}{\it{Let $X$ be a DG-module of finite G-dimension with $\mathrm{amp}X=\mathrm{amp}A$. Then we have \begin{center}$\mathrm{G}\textrm{-}\mathrm{dim}_AX+\mathrm{sup}X=\mathrm{seq.depth}A-\mathrm{seq.depth}_AX$.\end{center}}}
\end{cor}
\begin{proof} One has the following equalities
\begin{center}$\begin{aligned}\mathrm{G}\textrm{-}\mathrm{dim}_AX+\mathrm{sup}X
&=\mathrm{depth}A-\mathrm{depth}_AX+\mathrm{sup}X\\
&=\mathrm{depth}A-\mathrm{inf}A-(\mathrm{depth}_AX-\mathrm{inf}X)\\
&=\mathrm{seq.depth}A-\mathrm{seq.depth}_AX,
\end{aligned}$\end{center}where the first one is by Theorem \ref{lem:3.3}
and the second one is by $\mathrm{amp}X=\mathrm{amp}A$.
\end{proof}

We are now ready to prove Corollary \ref{the3.4}.

{\bf Proof of Corollary \ref{the3.4}.} (1) The ``only if" part holds by noting that $A\in\mathcal{H}$ by \cite[Theorem 4.1]{Shaul}. For the ``if" part, we assume that $X$ is an object in $\mathcal{H}$. It follows that $\mathrm{lc.~dim}_AA\geq\mathrm{lc.dim}_AX-\mathrm{sup}X\geq\mathrm{depth}_AX-\mathrm{inf}X$. Since
 $A$ is local Cohen-Macaulay, we have the following (in)equalities
\begin{center}$\begin{aligned}\mathrm{seq.depth}_AX
&\leq\mathrm{lc.dim}_AX-\mathrm{sup}X\\
&\leq\mathrm{lc.dim}A\\
&=\mathrm{seq.depth}A\\
&=\mathrm{seq.depth}_AX,
\end{aligned}$\end{center}
where the second equality follows from Corollary \ref{lem:3.3'}, which implies that
$\mathrm{ampR}\Gamma_{\bar{\mathfrak{m}}}X=\mathrm{amp}A=\mathrm{amp}X$ and $\mathrm{lc.dim}_AX=\mathrm{dim}\mathrm{H}^{0}(A)+\mathrm{sup}X$. Thus $X\in\mathcal{M}$.

(2) For the ``only if" part, since $A$ is local Gorentein, it follows from \cite[Theorem 7.26]{s18} that $X\simeq\mathrm{RHom}_{A}(\mathrm{RHom}_{A}(X,A),A)$ and $\mathrm{ampR}\Gamma_{\bar{\mathfrak{m}}}X=\mathrm{amp}\mathrm{RHom}_{A}(X,A)$ for any $X\in\mathrm{D}^{\mathrm{b}}_{\mathrm{f}}(A)$. Let $X\in\mathcal{M}$. Then $\mathrm{amp}\mathrm{RHom}_{A}(X,A)=\mathrm{amp}A$, it follows from Corollary \ref{lem:3.3'} that
 \begin{center}$\mathrm{G}\textrm{-}\mathrm{dim}_AX+\mathrm{sup}X
=\mathrm{seq.depth}A-\mathrm{seq.depth}_AX=0$.\end{center}
Thus $X\in\mathcal{H}$ and $\mathcal{H}=\mathcal{M}$ by (1).

For the ``if" part, we assume that $\mathcal{H}=\mathcal{M}$.
Then $A$ is local-Cohen-Macaulay by (1). Consider the sppj resolution $P^\bullet$ of the DG-module $k:=\bar{k}\oplus\bar{k}[-\mathrm{inf}A]$
\begin{center}$k_{d}\rightarrow P_{d-1}\rightarrow\cdots P_{1}\rightarrow P_{0}\rightarrow k$.\end{center}
By the proof of \cite[Proposition 4.4]{YL}, one has
 $G=A[-\mathrm{sup}k_d]\oplus k_d\in\mathcal{M}$ and $k$ han an sppj resolution
\begin{center}$G\rightarrow A[-\mathrm{sup}k_d]\oplus P_{d-1}\rightarrow\cdots P_{1}\rightarrow P_{0}\rightarrow k$.\end{center}Thus, G-dim$_Ak< \infty$ by Theorem \ref{thm:main}(1) and hence $A$ is Gorenstein.\hfill$\Box$

\subsection{Buchweitz-Happel Theory and its inverse}\label{subsection:4.3}
In this subsection we will employ Theorem \ref{thm:main} to prove Corollary \ref{cor1.6} in the introduction.

Let $\mathcal{T}$ be a triangulated category and $\mathcal{D}\subseteq \mathcal{C}$ be subcategories in $\mathcal{T}$. As usual, we denote by $ \langle\mathcal{D}\rangle$ the smallest thick triangulated subcategory containing $\mathcal{D}$, and denote by $\mathcal{C}/\langle\mathcal{D}\rangle$ the Verdier's quotient triangulated category.
In fact, $\langle\mathcal{P}\rangle$ is the thick subcategory of the derived category $\textrm{D}^{\mathrm{b}}(A)$ generated by $A$. We call the Verdier quotient $ \textrm{D}_{\mathrm{sg}}(A):=\textrm{D}_{\mathrm{f}}^{\mathrm{b}}(A)/\langle\mathcal{P}\rangle$ the \emph{singularity category} of $A$.
Each morphism $f : X\rightarrow Y$  in $\textrm{D}_{\mathrm{sg}}(A)$ is given
by an equivalence class of left fractions $s\setminus f$ as presented by $s\setminus f: X \stackrel{f}\longrightarrow Z\stackrel{s}\Longleftarrow Y$ , where the doubled arrow means $s$  lies in the compatible saturated multiplicative system corresponding to $\langle\mathcal{P}\rangle$.

 In the following, we set
$$A[< 0]^{\perp_{f}} = \{X \in \mathcal{R}(A)\mid\mathrm{Hom}_{\mathrm{D}(A)}(A[< 0],X) = 0\};$$
$$^{\perp_{f}}A[>0] = \{X \in \mathcal{R}(A) \mid\mathrm{Hom}_{\mathrm{D}(A)}(X, A[>0]) = 0\}.$$

Recall from Corollary \ref{the3.4} that
$$\mathcal{H}=\{X\in\mathcal{G}\mid
\mathrm{ampR}\Gamma_{\bar{\mathfrak{m}}}X\geq\mathrm{amp}X=\mathrm{amp}A\}.$$

\begin{lem}\label{lem} Let $X\in \mathcal{R}(A)$. Assume that
 $$\mathrm{ampR}\Gamma_{\bar{\mathfrak{m}}}X\geq\mathrm{amp}X\geq \mathrm{amp}A.$$
 Then $X\in \mathcal{H}$ if and only if $\sup X[\sup X]\leq0$ and $\mathrm{G}\text{-}\mathrm{dim}_AX[\sup X]\leq0$ if and only if $X[\sup X]\in  {^{\perp_{f}}{A[>0]}}\cap{{A[<0]}^{\perp_{f}}}$.
\end{lem}
\begin{proof} We will start by proving the first equivalence. Let $X\in\mathcal{H}$.  Since $X\in \mathcal{G}$, $\mathrm{G}\text{-}\mathrm{dim}_AX = - \sup X$. Therefore $\mathrm{G}\text{-}\mathrm{dim}_AX[\sup X]= \mathrm{G}\text{-}\mathrm{dim}_AX + \sup X= 0$ and $\sup X[\sup X]\leq0$.
Conversely, suppose that $\sup X[\sup X]\leq0$ and $\mathrm{G}\text{-}\mathrm{dim}_AX[\sup X]\leq0$. Then $\mathrm{G}\text{-}\mathrm{dim}_AX\leq -\sup X$ as $\mathrm{G}\text{-}\mathrm{dim}_AX[\sup X]=\mathrm{G}\text{-}\mathrm{dim}_AX + \sup X$.  By Corollary \ref{lem0.3}, we have $\mathrm{G}\text{-}\mathrm{dim}_AX\geq -\sup X$. It follows that $X\in \mathcal{G}$.  In order to show $X\in\mathcal{H}$, it
suffices to show $\mathrm{amp}X\leq \mathrm{amp}A$.  Since  $X[\sup X]\in \mathcal{R}(A)$, by \cite[Lemma 3.2(i)]{BSSW}, we have
$$\inf(X[\sup X])=\mathrm{inf}~\mathrm{RHom}_{A}(\mathrm{RHom}_{A}(X[\sup X],A),A)
\geq-\mathrm{sup}\mathrm{RHom}_{A}(X[\sup X],A)+\mathrm{inf}(A).$$
Then
$0\geq\mathrm{G}\text{-}\mathrm{dim}_AX[\sup X]=\sup(\mathrm{RHom}_{A}(X[\sup X],A))\geq\inf(A)-\inf(X[\sup X])$,
it follows that $-\inf(A)\geq -\inf(X[\sup X])$. On the other hand, since $\sup A=0$ and $0\geq \sup X[\sup X]$, we have
$\mathrm{amp} A=-\inf(A)\geq \sup X[\sup X]-\inf(X[\sup X])=\mathrm{amp}X[\sup X]= \mathrm{amp}X$, as desired.

The second equivalence follows from the equalities
$^{\perp_{f}}A[>0]  =\{X \in \mathcal{R}(A)\mid\mathrm{H}^{i}(\mathrm{RHom}_{A}(X,A)) = 0~\textrm{for}~i>0\}$
and $A[< 0]^{\perp_{f}} =\{X \in \mathcal{R}(A)\mid\mathrm{H}^{i}(\mathrm{RHom}_{A}(A,X)) = 0~\textrm{for}~i>0\}
=\{X \in \mathcal{R}(A)\mid\mathrm{H}^{i}(X)= 0~\mathrm{for}~i>0\}$.
\end{proof}
In the following, we set
$\mathcal{A} = {^{\perp_{f}}{A[>0]}}\cap{{A[<0]}^{\perp_{f}}}.$

\begin{remark} \label{rem4.6} {\rm
Let $A$ be a non-positive DG-algebra over a field $k$ satisfying the following conditions:
\begin{enumerate}
\item[(i)] $A$ is proper, i.e., $\mathrm{dim}_{k} \bigoplus_{i \in \mathbb{Z}} \mathrm{H}^{i}(A) < \infty$;

\item[(ii)] $A$ is Gorenstein, i.e., the thick subcategory $\mathrm{per}~A$ of the derived category $\mathrm{D}(A)$ generated by $A$ coincides with the thick subcategory generated by $DA$, where $D = \Hom_{k}(-, k)$ is the $k$-dual.
\end{enumerate}

Assume that $A$ is a proper Gorenstein DG-algebra. Recently, Jin defined that a DG-$A$-module $X\in\mathrm{D}^{\mathrm{b}}(A)$ is called \emph{Cohen-Macaulay} if $\sup M\leq 0 $ and $\Hom_{\mathrm{D}^{\mathrm{b}}(A)}(M,A[i]) =0$  for all $i >0$
(see \cite[Definition 2.1]{jin20}). If $A$ is an ordinary Gorenstein algebra, then Cohen-Macaulay DG-modules defined by Jin is canonically equivalent to the usual (maximal) Cohen-Macaulay modules. Note that if a DG-module $X$ belongs to $\mathcal{R}(A)$, then $X\in\mathcal{A}$ if and only if it is Cohen-Macaulay defined by Jin in the sense of \cite{jin20}.

On the other hand, from Corollary \ref{the3.4} we know that $\mathcal{H}$ is equal to the class of maximal local-Cohen-Macaulay DG-modules defined by Shaul \cite{Shaul} over local Gorenstein DG-rings, see Definition \ref{def}.  These two definitions are completely different at first glance. For example, one can see that the class of maximal local-Cohen-Macaulay DG-modules defined by Shaul is closed under suspensions, but the class of Cohen-Macaulay DG-modules defined by Jin is not. Fortunately, by Lemma \ref{lem}, these two definitions are the same over commutative local Gorenstein DG-algebra over a field $k$ under the condition that $\mathrm{ampR}\Gamma_{\bar{\mathfrak{m}}}X\geq\mathrm{amp}X\geq \mathrm{amp}A$ for all $X\in{\mathcal{R}(A)}$. {This condition is natural and very often met. For example, this holds for any complex $X\in\mathrm{D}^{\mathrm{b}}_{\mathrm{f}}(A)$ over a commutative noetherian local ring $A$ by [31, Propositions 2.8 and 3.4]. Moreover, we assume that $A$ is a DG-ring with $\mathrm{dim}\mathrm{H}^0(A)=0$. Let $X\in\mathrm{D}^{\mathrm{b}}_{\mathrm{f}}(A)$ with $\mathrm{amp}X\geq \mathrm{amp}A$. Then $\mathrm{supR}\Gamma_{\bar{\mathfrak{m}}}X\geq\mathrm{sup}X$ by [31, Remark 2.3 and Theorem 2.15] and $\mathrm{infR}\Gamma_{\bar{\mathfrak{m}}}X\leq\mathrm{inf}X$ by [31, Propositions 3.3 and 3.5]. So we have $\mathrm{ampR}\Gamma_{\bar{\mathfrak{m}}}X\geq\mathrm{amp}X$, as desired.}}
\end{remark}

\begin{lem}\label{lem6.2} The subcategory $\mathcal{R}(A)$ is thick and it is the smallest triangulated subcategory of $\mathrm{D}^{\mathrm{b}}_{\mathrm{f}}(A)$ that contains the objects from $\mathcal{A}$.
\end{lem}
\begin{proof} Obviously, $\mathcal{R}(A)$ is closed under shifts and isomorphisms.
Assume that $X\in{\mathcal{R}(A)}$ and $X'$ is a direct summand of $X$. There
exits two split triangles $X'\s{u}\rightarrow X\s{v}\rightarrow X''\s{0}\rightarrow X'[1]$ and
$X''\s{p}\rightarrow X\s{q}\rightarrow X'\s{0}\rightarrow X''[1]$ in $\mathrm{D}^{\mathrm{b}}_{\mathrm{f}}(A)$. For simplicity, we write $((-, A),A)$ for $\mathrm{RHom}_{A}(\mathrm{RHom}_{A}(-,A),A)$. Then we have
the following commutative diagrams of exact triangles:
\begin{equation}\label{3.4-1}
\begin{aligned}
\xymatrix@C=20pt@R=18pt{
  X' \ar[d]_{\delta^{X'}} \ar[r]^{u} & X \ar[d]_{\delta^{X}} \ar[r]^{v} & X'' \ar[d]_{\delta^{X''}} \ar[r]^{0} & X'[1] \ar[d]_{\delta^{X'}[1]} \\
  ((X', A),A) \ar[r]^{} &((X, A),A) \ar[r] & ((X'', A),A) \ar[r]^{0\ \ } & ((X', A),A)[1],}
  \end{aligned}
\end{equation}
\begin{equation}\label{3.4-2}
\begin{aligned}
\xymatrix@C=20pt@R=18pt{
  X'' \ar[d]_{\delta^{X''}} \ar[r]^{p} & X \ar[d]_{\delta^{X}} \ar[r]^{q} & X' \ar[d]_{\delta^{X'}} \ar[r]^{0} & X''[1] \ar[d]_{\delta^{X''}[1]} \\
  ((X'', A),A) \ar[r]^{} &((X, A),A) \ar[r]^{} & ((X', A),A) \ar[r]^{0\ \ } & ((X'', A),A)[1].}
  \end{aligned}
\end{equation}
Since $u$ is a split monomorphism and $\delta^{X}$ is an isomorphism, it is straightforward to check that $\delta^{X'}$ is a split monomorphism by the commutative diagram \eqref{3.4-1}. On the other hand,  it follows from the commutativity of the diagram \eqref{3.4-2} that $\delta^{X'}q$ is an epimorphism, which implies that $\delta^{X'}$ is also an epimorphism. Thus, $\delta^{X'}$ is an isomorphism and $X'$ belongs to $\mathcal{R}(A)$, and therefore $\mathcal{R}(A)$ is a thick subcategory of $\mathrm{D}^{\mathrm{b}}_{\mathrm{f}}(A)$ by Proposition \ref{lem3.5}.

Recall that $\langle\mathcal{A}\rangle$ is the smallest thick triangulated subcategory of $\mathrm{D}^{\mathrm{b}}_{\mathrm{f}}(A)$ that contains the objects from $\mathcal{A}$. Since $\mathcal{A}\subseteq \mathcal{R}(A)$, it follows that $\langle\mathcal{A}\rangle\subseteq\mathcal{R}(A)$. For the reverse containment, we assume that $X$ is a non-zero object in $\mathcal{R}(A)$. Note that $\mathcal{R}(A)$ is precisely the subcategory of DG-modules with finite G-dimension. If $\mathrm{amp}X\geq\mathrm{amp}A$, applying Theorem \ref{thm:main}(1), there exists a natural number $m$ such that $X$ belongs to $\mathcal{G}_0[-\mathrm{sup}X]\ast\mathcal{P}[-\mathrm{sup}X+1]\ast\cdots\ast\mathcal{P}[-\mathrm{sup}X+m]$. By the fact that $\mathcal{P}\subseteq\mathcal{G}_0\subseteq \mathcal{A}$, we deduce that $X$ belongs to $\langle\mathcal{A}\rangle$.
If $\mathrm{amp}X<\mathrm{amp}A$, applying Theorem \ref{thm:main}(2), then
$X\oplus X[\mathrm{amp}A-\mathrm{amp}X]$ belongs to $\mathcal{P}[-\mathrm{sup}X]\ast\cdots\ast\mathcal{P}[-\mathrm{sup}X+n-1]\ast\mathcal{G}_0[-\mathrm{sup}X+n]$ for some $n$. Similarly, since $\langle\mathcal{A}\rangle$ is thick, one has $X\in\langle\mathcal{A}\rangle$. Altogether, we have $\mathcal{R}(A)=\langle\mathcal{A}\rangle$.
\end{proof}

Let $\mathcal{T}$ be a triangulated category. Recall that a subcategory $\omega\subseteq{\mathcal{T}}$ is called \emph{presilting} if $\mathrm{Hom}_{\mathcal{T}}(M,M'[i])=0$ for all $M, M'\in \omega$ and
$i>0$.
Let $\mathcal{E}$ be an arbitrary category. A full additive subcategory $\mathcal{M} \subseteq \mathcal{E}$ is
called \emph{contravariantly finite}, if every object $X$ in $\mathcal{E}$ admits a right $\mathcal{M}$-approximation,
i.e. there exists an object $M \in\mathcal{ M}$ and a morphism $f : M \rightarrow X$, such that the induced
map $\Hom_{\mathcal{E}}(N, M)\rightarrow \Hom_{\mathcal{E}}(N, X)$ is surjective for all $N \in \mathcal{M}$. Dually, we define
the notion of a \emph{covariantly finite} subcategory. We say that $M$ is \emph{functorially finite}
if it is both covariantly and contravariantly finite.

\begin{lem} \label{lem3.4} The following statement hold.

\begin{enumerate}
\item $\mathcal{A}$ is closed under extensions and direct summands and finite direct sums in $\mathcal{R}(A)$.

\item The subcategory $\mathcal{P}=\mathrm{add}~A$ is  presilting and functorially finite in $\mathcal{R}(A)$.
\end{enumerate}
\end{lem}
\begin{proof}
(1) Since $\mathcal{A}={{A[<0]}^{\perp_{f}}}\cap{^{\perp_{f}}{A[>0]}}$, one can check that all objects in $\mathcal{A}$ is closed under extensions and direct summands and finite direct sums in $\mathcal{R}(A)$.

(2) The subcategory $\mathcal{P}$ is  presilting follows from the fact that $A$ is a non-positive DG-ring.
Note that for any $N\in \mathcal{R}(A)$, the set of the generators of $\mathrm{Hom}_{\mathrm{D}(A)}(A,N) \cong \mathrm{H}^{0}(N)$  is finite.
If $f_{1},\cdots,f_{s}$ are generators of $\mathrm{Hom}_{\mathrm{D}(A)}(A,N)$, then the canonical homomorphism $f : A^{(s)}\rightarrow N$ given by $(f_{1},\cdots,f_{s})$  is a right $\mathcal{P}$-approximation.
On the other hand, since $\mathrm{RHom}_{A}(X,A)\in\mathrm{D}^{\mathrm{b}}_{\mathrm{f}}(A)$ for any $X\in\mathcal{R}(A)$, the set
$\mathrm{Hom}_{\mathrm{D}(A)}(X,A)$ also has finite generators. So the left $\mathcal{P}$-approximation can be constructed by a similar way.
Therefore, $\mathcal{P}$ is functorially finite in $\mathcal{R}(A)$.
\end{proof}

 Following \cite{we16}, we denote by $[\mathcal{D}]$ the ideal of morphisms in $\mathcal{C}$ which factor through objects in $\mathcal{D}$. The stable category of $\mathcal{C}$ by $\mathcal{D}$ is the quotient category $\mathcal{C}/[\mathcal{D}]$, for which objects are the same as $\mathcal{C}$, but morphisms are morphisms in $\mathcal{C}$ modulo those in $[\mathcal{D}]$.

The following results are analogues of well-known properties of
Cohen-Macaulay modules.

\begin{thm}\label{thm4.9}\begin{enumerate}
\item The stable category $\underline{\mathcal{A}}:=\mathcal{A}/[\mathcal{P}]$ is a triangulated category;

\item The composition  $\mathcal{A}\hookrightarrow \mathcal{R}(A)\rightarrow \mathcal{R}(A)/ \langle\mathcal{P}\rangle$ induces a triangle equivalence
$$H:~~~ \underline{\mathcal{A}}\simeq \mathcal{R}(A)/ \langle\mathcal{P}\rangle.$$

\item $\mathcal{R}(A)=\mathrm{D}^{\mathrm{b}}_{\mathrm{f}}(A)$ if and only if $A$ is a local Gorenstein DG-ring.
\end{enumerate}
\end{thm}
\begin{proof}
$(1)$ By Lemma \ref{lem3.4}(2), we know that $\mathcal{P}$ is  presilting  and functorially finite in $\mathcal{R}(A)$.
Then the result follows from the proof of \cite[Corollary 2.7]{we16} and \cite[Proposition 2.4(3)]{we16}.

$(2)$  Since $\mathcal{P}\subseteq \mathcal{R}(A)$ and $\mathcal{R}(A)$ is a triangulated full subcategory of $\mathrm{D}^{\mathrm{b}}_{\mathrm{f}}(A)$, we have
$\langle\mathcal{P}\rangle \subseteq \mathcal{R}(A)$. Thanks to \cite[Corollary 2.7]{we16} or \cite[Theorem 1.1]{ya13}, by Lemma  \ref{lem3.4}, there is a triangle equivalence between $\underline{\mathcal{A}}$ and the singularity category $\mathrm{D}_{\mathrm{f}}^{\mathrm{b}}(A)/ \langle\mathcal{P}\rangle$.

$(3)$ This follows from Theorem \ref{thm:main}(3).
\end{proof}

Let $G$ be the obvious composite functor $\mathcal{A}\rightarrow \mathrm{D}_{\mathrm{f}}^{\mathrm{b}}(A)\rightarrow \mathrm{D}_{\mathrm{f}}^{\mathrm{b}}(A)/ \langle\mathcal{P}\rangle$. Since $G(\mathcal{P}) = 0$, there is a unique
factorization of $G$ through the canonical projection $S : \mathcal{A} \rightarrow \underline{\mathcal{A}}$. We denote by $F$ the factorization from $\underline{\mathcal{A}}$ to $\mathrm{D}_{\mathrm{f}}^{\mathrm{b}}(A)/ \langle\mathcal{P}\rangle$. Then $G = F\circ S$. In fact, $F$ is the  composite functor $\underline{\mathcal{A}}\overset{H}{\rightarrow} \mathcal{R}(A)/ \langle\mathcal{P}\rangle\hookrightarrow \mathrm{D}_{\mathrm{f}}^{\mathrm{b}}(A)/ \langle\mathcal{P}\rangle$.
Let $A$ be a local Gorenstein DG-ring. By Theorem \ref{thm4.9}(3), we obtain a triangle equivalence:
$F: \underline{\mathcal{A}}\simeq \mathrm{D}_{\mathrm{f}}^{\mathrm{b}}(A)/ \langle\mathcal{P}\rangle.$

\medskip

We are now ready to prove Corollary \ref{cor1.6}.

{\bf Proof of Corollary \ref{cor1.6}.} By Theorem \ref{thm4.9}, it suffices to show that if the functor $F: \underline{\mathcal{A}}\rightarrow \mathrm{D}_{\mathrm{f}}^{\mathrm{b}}(A)/ \langle\mathcal{P}\rangle$ is dense, then  $A$ is a local Gorenstein DG-ring.
 Assume that $F: \underline{\mathcal{A}}\rightarrow \mathrm{D}_{\mathrm{f}}^{\mathrm{b}}(A)/ \langle\mathcal{P}\rangle$ is dense, and let $M$ be any DG-module in $\mathrm{D}_{\mathrm{f}}^{\mathrm{b}}(A)$. It follows that $M\cong F (G)$ in $\mathrm{D}_{\mathrm{f}}^{\mathrm{b}}(A)$ for some $G\in \mathcal{A}$. Let $s\setminus f: M \stackrel{f}\longrightarrow Z\stackrel{s}\Longleftarrow G$ be an isomorphism in $\mathrm{D}_{\mathrm{f}}^{\mathrm{b}}(A)$ with $\mathrm{cn}(s)\in \langle\mathcal{P}\rangle$, then $\mathrm{cn}(f)\in \langle\mathcal{P}\rangle$. Consider the triangle $G \stackrel{s}\rightarrow Z\longrightarrow \mathrm{cn}(s)\rightsquigarrow$ in $\mathrm{D}_{\mathrm{f}}^{\mathrm{b}}(A)$. We see that $G$ and $\mathrm{cn}(s)$ lie in $\mathcal{R}(A)$, so $Z\in \mathcal{R}(A)$. It follows from $Z\in \mathcal{R}(A)$ and $\mathrm{cn}(f)\in \mathcal{R}(A)$ that $M\in \mathcal{R}(A)$. Therefore the  G-dimension of $M$ is finite, and so $A$ is a local Gorenstein DG-ring by Theorem \ref{thm:main}(3).
\hfill$\Box$


\begin{remark}\label{remark:4.12} {\rm
If $A$ is a local Gorenstein ring, then Corollary \ref{cor1.6} gives the classical Buchweitz-Happle Theorem and its inverse.
As we stated in the Introduction, Jin \cite{jin20} proved that if $A$ is a proper Gorenstein DG-algebra, then there is a triangle equivalence between the stable category of Cohen-Macaulay DG-modules (see Remark \ref{rem4.6}) and the singularity category
$\mathrm{D}^{\mathrm{b}}(A)/\langle\mathcal{P}\rangle$. This generalizes Buchweitz-Happle Theorem to the setting of proper Gorenstein DG-algebras (see \cite[Theorem 2.4]{jin20}). In fact, over commutative local noetherian DG-algebras over a field $k$, condition (ii) of the proper Gorenstein DG-algebra is equivalent to $\mathrm{injdim}_AA < \infty$. In this case,  a Gorenstein ring in our setting is a proper Gorenstein DG-algebra defined in \cite{jin20}. Therefore, over a commutative local noetherian non-positive DG algebra over $k$, Corollary \ref{cor1.6} shows that the inverse of Buchweitz-Happel Theorem is also true.}
\end{remark}

\smallskip
{\bf Acknowledgements.} The research was supported by the National Natural Science Foundation of China (Grant Nos. 12171206, 12201223). The authors would like to thank Dong Yang for sharing Example 3.4, and Haibo Jin for helpful discussions, particularly regarding Cohen-Macaulay DG-modules and the Buchweitz-Happel Theorem. Both authors are very grateful to the anonymous referee for suggestions on the language and exposition of the article.

\bigskip \centerline {\bf Declaration of Interest Statement}
The authors declare that they have no known competing financial interests or  personal relationships that could have appeared to influence the work reported in this article.

\renewcommand\refname{References}

\vspace{4mm}
\noindent\textbf{Jiangsheng Hu}\\
School of Mathematics, Hangzhou Normal University, Hangzhou 311121, P. R. China.\\
Email: \textsf{hujs@hznu.edu.cn}\\[1mm]
\textbf{Xiaoyan Yang}\\
School of Science, Zhejiang University of Science and Technology, Hangzhou 310023, P. R. China.\\
Email: \textsf{yangxy@zust.edu.cn}\\[1mm]
\textbf{Rongmin Zhu}\\
School of Mathematical
Sciences, Huaqiao University, Quanzhou 362021, P. R. China.\\
E-mail: \textsf{rongminzhu@hotmail.com}\\[1mm]
\end{document}